\documentclass[12pt]{amsart}

\usepackage{amsmath, amssymb,amsxtra}
\usepackage{amsthm} 
\newtheorem*{previousthmA}{Theorem A}
\newtheorem*{previousthmB}{Theorem B}
\newtheorem*{previousthmC}{Theorem C}
\newtheorem*{previousthmD}{Theorem D}
\newtheorem*{previousthmE}{Theorem E}

\theoremstyle{plain} 
\newtheorem{theorem}{\sc Theorem}[section] 
\newtheorem{lemma}[theorem]{\sc Lemma}
\newtheorem{corollary}[theorem]{\sc Corollary}

\theoremstyle{definition} 

\newtheorem{remark}[theorem]{\sc Remark}


\makeatletter 

\title[]{The inclusion relation between Sobolev and modulation spaces}
\author{Masaharu Kobayashi and Mitsuru Sugimoto}
\address{Masaharu Kobayashi \\
Department of Mathematics, 
Tokyo University of Science, 
Kagurazaka 1-3, Shinjuku-ku, Tokyo 162-8601, Japan}
\email{kobayashi@jan.rikadai.jp }

\address{Mitsuru Sugimoto \\
Graduate School of Mathematics \\
Nagoya University, \\
Furocho, Chikusa-ku, Nagoya 464-8602, Japan}
\email{sugimoto@math.nagoya-u.ac.jp}

\makeatother
\date{\today}

\begin{document}
\maketitle

\begin{abstract}
The inclusion relations between the $L^p$-Sobolev spaces and
the modulation spaces is determined explicitly.
As an application, mapping properties of unimodular Fourier multiplier
$e^{i|D|^\alpha}$ between $L^p$-Sobolev spaces and modulation spaces
are discussed.
\end{abstract}


\section{Introduction}

The modulation spaces $M^{p,q}_s$ are one of the function spaces
introduced by Feichtinger \cite{Feichtinger} in 1980's 
to measure the decaying and regularity property
of a function or distribution in a
way different from $L^p$-Sobolev spaces $L^p_s$ or  Besov spaces $B^{p,q}_s$.
The precise definitions of these function spaces will be given in
Section \ref{Preliminaries}, but the main idea of modulation spaces is to
consider the space variable and the
variable of its Fourier transform simultaneously,
while they are treated independently in $L^p$-Sobolev spaces and Besov spaces.

Because of this special nature, modulation spaces are now considered
to be suitable spaces in the analysis of 
pseudo-differential operators
after a series of important works \cite{Cordero-Grochenig}, \cite{Grochenig 2006},
\cite{Grochenig-Heil}, \cite{Hei-Ramanathan-Topiwala}, \cite{Tachizawa},
\cite{Toft} 
and so on.
On the other hand, modulation spaces have also remarkable applications
in the analysis of partial differential equations.
For example, the Schr\"odinger and wave propagators, which are not bounded
on neither $L^p$ nor $B^{p,q}_s$, are bounded on $M_s^{p,q}$
(\cite{Benyi et all}).
Modulation spaces are also used as a regularity class
of initial data of the Cauchy problem for nonlinear evolution equations,
and in this way the existence of the solution is shown under very low
regularity assumption for initial data
(see \cite{Wang Huang}, \cite{Wang-Hudzik}, \cite{WZG}).

In the last several years,
many basic properties of modulation spaces are established.
In particular, the inclusion relation between Besov spaces and modulation
spaces has been completely determined.
Let us define the indices $\nu_1(p,q)$ and $\nu_2(p,q)$ 
for $1\leq p,q \leq \infty$ in the following way:
\begin{align*}
\nu_1(p,q)=
\begin{cases}
0 & {\rm if~} (1/p,1/q) \in I^*_1 ~:~ \min(1/p,1/p^\prime) \geq 1/q, \\
1/p+1/q -1 & {\rm if~} (1/p,1/q) \in I^*_2 ~:~  \min(1/q,1/2) \geq 1/p^\prime, \\
-1/p +1/q  & {\rm if~} (1/p,1/q) \in I^*_3  ~:~\min (1/q,1/2) \geq 1/p,
\end{cases}\\
\nu_2(p,q)=
\begin{cases}
0 & {\rm if~} (1/p,1/q) \in I_1 ~:~  \max(1/p,1/p^\prime) \leq 1/q, \\
1/p+1/q -1 & {\rm if~} (1/p,1/q) \in I_2 ~:~  \max(1/q,1/2) \leq 1/p^\prime, \\
-1/p +1/q  & {\rm if~} (1/p,1/q) \in I_3  ~:~\max(1/q,1/2) \leq 1/p,
\end{cases}
\end{align*}
where $1/p+1/p^\prime=1 = 1/q +1/q^\prime.$
We remark $\nu_2(p,q)=-\nu_1(p',q')$.

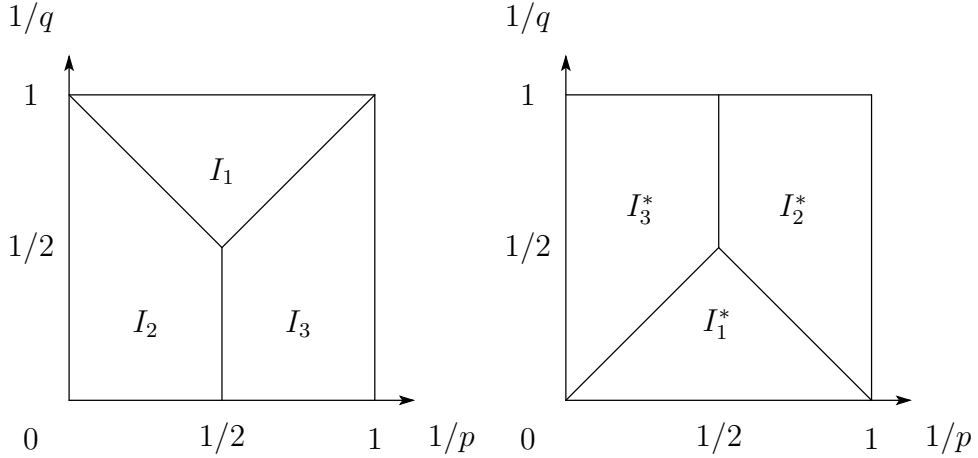
$$
\unitlength 0.1in
\begin{picture}( 48.2500, 22.0000)( -0.2500,-25.1500)
%
\special{pn 8}%
\special{pa 400 2400}%
\special{pa 400 600}%
\special{fp}%
\special{sh 1}%
\special{pa 400 600}%
\special{pa 380 668}%
\special{pa 400 654}%
\special{pa 420 668}%
\special{pa 400 600}%
\special{fp}%
\special{pa 400 2400}%
\special{pa 2200 2400}%
\special{fp}%
\special{sh 1}%
\special{pa 2200 2400}%
\special{pa 2134 2380}%
\special{pa 2148 2400}%
\special{pa 2134 2420}%
\special{pa 2200 2400}%
\special{fp}%
\special{pa 3000 2400}%
\special{pa 3000 600}%
\special{fp}%
\special{sh 1}%
\special{pa 3000 600}%
\special{pa 2980 668}%
\special{pa 3000 654}%
\special{pa 3020 668}%
\special{pa 3000 600}%
\special{fp}%
\special{pa 3000 2400}%
\special{pa 4800 2400}%
\special{fp}%
\special{sh 1}%
\special{pa 4800 2400}%
\special{pa 4734 2380}%
\special{pa 4748 2400}%
\special{pa 4734 2420}%
\special{pa 4800 2400}%
\special{fp}%
%
\special{pn 8}%
\special{pa 4600 2400}%
\special{pa 4600 800}%
\special{fp}%
\special{pa 4600 800}%
\special{pa 3000 800}%
\special{fp}%
\special{pa 2000 800}%
\special{pa 2000 2400}%
\special{fp}%
\special{pa 2000 800}%
\special{pa 400 800}%
\special{fp}%
\special{pa 400 800}%
\special{pa 1200 1600}%
\special{fp}%
\special{pa 1200 1600}%
\special{pa 2000 800}%
\special{fp}%
\special{pa 1200 1600}%
\special{pa 1200 2400}%
\special{fp}%
%
\special{pn 8}%
\special{pa 3000 2400}%
\special{pa 3800 1600}%
\special{fp}%
\special{pa 3800 1600}%
\special{pa 3800 800}%
\special{fp}%
\special{pa 3800 1600}%
\special{pa 4600 2400}%
\special{fp}%
\put(38.0000,-26.0000){\makebox(0,0){$1/2$}}%
\put(46.0000,-26.0000){\makebox(0,0){$1$}}%
\put(28.0000,-26.0000){\makebox(0,0){$0$}}%
\put(28.0000,-16.0000){\makebox(0,0){$1/2$}}%
\put(28.0000,-8.0000){\makebox(0,0){$1$}}%
\put(28.0000,-4.0000){\makebox(0,0){$1/q$}}%
\put(50.0000,-26.0000){\makebox(0,0){$1/p$}}%
\put(2.0000,-4.0000){\makebox(0,0){$1/q$}}%
\put(2.0000,-8.0000){\makebox(0,0){$1$}}%
\put(2.0000,-16.0000){\makebox(0,0){$1/2$}}%
\put(2.0000,-26.0000){\makebox(0,0){$0$}}%
\put(12.0000,-26.0000){\makebox(0,0){$1/2$}}%
\put(20.0000,-26.0000){\makebox(0,0){$1$}}%
\put(24.0000,-26.0000){\makebox(0,0){$1/p$}}%
\put(12.0000,-12.0000){\makebox(0,0){$I_1$}}%
\put(8.0000,-20.0000){\makebox(0,0){$I_2$}}%
\put(16.0000,-20.0000){\makebox(0,0){$I_3$}}%
\put(37.9000,-20.0000){\makebox(0,0){$I^*_1$}}%
\put(41.9000,-14.0000){\makebox(0,0){$I_2^*$}}%
\put(33.9000,-14.0000){\makebox(0,0){$I^*_3$}}%
\end{picture}%
$$
\vspace{0.5cm}

\noindent
Then the following result is known:
\begin{theorem}[Sugimoto-Tomita \cite{Sugimoto-Tomita}, Toft \cite{Toft}]
\label{BpqMpq}
Let $1 \leq p,q \leq \infty$ and $s \in {\mathbf R}$.
Then we have
\begin{enumerate}
\item[$(1)$] $B^{p,q}_s({\mathbf R}^n) \hookrightarrow M^{p,q}({\mathbf R}^n)$
if and only if $s \geq n \nu_1(p,q);$

\item[$(2)$]  $M^{p,q}({\mathbf R}^n) \hookrightarrow B^{p,q}_s({\mathbf R}^n)$
if and only if $s \leq n \nu_2(p,q)$.

\end{enumerate}

\end{theorem}

As for the inclusion relation between $L^p$-Sobolev spaces and
modulation spaces, the following result (see also \cite{Toft 155})
is immediately obtained from
Theorem \ref{BpqMpq} if we notice the inclusion
property $L^p_{s+\varepsilon}\hookrightarrow
B^{p,q}_s\hookrightarrow L^p_{s-\varepsilon}$ for $\varepsilon>0$
(see, \cite[p.97]{Triebel II}):

\begin{corollary}
\label{LpMpq}
Let $1 \leq p,q \leq \infty$ and $s \in {\mathbf R}$.
Then we have
\begin{enumerate}
\item[$(1)$]  $L^{p}_s({\mathbf R}^n) \hookrightarrow M^{p,q}({\mathbf R}^n)$
if $s > n \nu_1(p,q)$.
Conversely,
if $L^{p}_s({\mathbf R}^n) \hookrightarrow M^{p,q}({\mathbf R}^n)$,
then $s \geq n \nu_1(p,q);$

\item[$(2)$]  $M^{p,q}({\mathbf R}^n) \hookrightarrow L^{p}_s({\mathbf R}^n)$
if $s < n \nu_2(p,q)$.
Conversely, if
$M^{p,q}({\mathbf R}^n) \hookrightarrow L^{p}_s({\mathbf R}^n)$,
then $s \leq n \nu_2(p,q)$. 
\end{enumerate}
\end{corollary}

But in Corollary \ref{LpMpq},
there still remains a question whether the critical case
$s=n\nu_1(p,q)$ or $s=n\nu_2(p,q)$ is sufficient or not for the inclusion.
The objective of this paper is to answer this basic question
and complete the picture of inclusion relations between the $L^p$-Sobolev
spaces and the modulation spaces.
The following theorems are our main results:

\begin{theorem}\label{MAIN THEOREM'}
Let $1 \leq p,q \leq \infty$ and $s \in {\mathbf R}.$
Then $ L^p_s({\mathbf R}^n) \hookrightarrow M^{p,q}({\mathbf R}^n)$
if and only if one of the following conditions is satisfied$:$
\begin{enumerate}
\item[$(1)$]  $q \geq  p> 1$ and $s \geq n \nu_1(p,q);$

\item[$(2)$]  $p>q$ and $s> n \nu_1(p,q);$

\item[$(3)$]  $p=1, q=\infty$, and $s \geq n \nu_1(1,\infty);$

\item[$(4)$]  $p=1, q \not=\infty$ and $s> n \nu_1(1,q)$.

\end{enumerate}
\end{theorem}

\begin{theorem}\label{MAIN THEOREM}
Let $1 \leq p,q \leq \infty$ and $s \in {\mathbf R}.$
Then $M^{p,q}({\mathbf R}^n) \hookrightarrow L^p_s({\mathbf R}^n)$
if and only if one of the following conditions is satisfied$:$
\begin{enumerate}
\item[$(1)$]  $q \leq p< \infty$ and $s \leq n \nu_2(p,q);$

\item[$(2)$]  $p<q$ and $s< n \nu_2(p,q);$
 
\item[$(3)$]  $p=\infty, q=1$, and $s \leq n \nu_2(\infty,1);$

\item[$(4)$]  $p=\infty, q \not=1$, and $s< n \nu_2(\infty,q)$.

\end{enumerate}
\end{theorem}

It should be mentioned that
Kobayashi-Miyachi-Tomita \cite{Kobayashi Miyachi Tomita}
determines the inclusion relation between modulation spaces $M_s^{p,q}$
and local Hardy spaces $h^p$ for $0<p\le 1$.
Our main results extend this result to the case $p>1$
since we have $h^p=L^p$ then.
As a matter of fact, the proof of Theorems \ref{MAIN THEOREM'}
and \ref{MAIN THEOREM} 
heavily depends on the results and arguments established in 
\cite{Kobayashi Miyachi Tomita}.


As an application of our main theorems, we also consider 
mapping properties of unimodular Fourier multiplier
$e^{i|D|^\alpha}, \alpha \geq 0$ , which is a generalization
of wave ($\alpha=1$) and Schr\"odinger ($\alpha=2$) propagators.
See Corollaries \ref{cor5.2} and \ref{MAIN THEOREM 4} in Section 5.
As Theorem A and Theorem B there say, the operator
$e^{i|D|^\alpha}$ ($0\leq\alpha\leq2$) is bounded on modulation
spaces while not on $L^p$-Sobolev spaces.
Theorems \ref{MAIN THEOREM'} and \ref{MAIN THEOREM} help us to
understand what happen if we consider the operator
between $L^p$-Sobolev spaces and modulation spaces.

We explain the organization of this paper.
After the next preliminary section
devoted to the definitions and basic properties
of function spaces treated in this paper,
we give a proof of Theorem \ref{MAIN THEOREM} in Sections 3 and 4.
We remark that Theorem \ref{MAIN THEOREM'} is just the dual statement of
Theorem \ref{MAIN THEOREM}.
In Section 5, we consider mapping properties of unimodular
Fourier multipliers between $L^p$-Sobolev spaces and modulation spaces,
as well as those of invertible pseudo-differential operators.

\section{Preliminaries} \label{Preliminaries}
\subsection{Basic notation}
The following notation will be used throughout this article.
We write ${\mathcal{S}}({\mathbf{R}}^n)$ to denote
the Schwartz space of all
complex-valued rapidly decreasing infinitely differentiable
functions on ${\mathbf{R}}^n$  and 
${\mathcal{S}}^{\prime}({\mathbf{R}}^n)$
to denote the space of tempered distributions on
${\mathbf R}^n$, i.e.,
the topological dual of
${\mathcal{S}}({\mathbf{R}}^n)$.
The Fourier transform is defined by
$\widehat{f}(\xi)  = \int_{{\mathbf R}^n} f(x) 
e^{- i x \cdot \xi}dx$ and the
inverse Fourier transform by ${f}^\vee(x) =(2 \pi)^{-n} \widehat{f}(-x)$.  
We define
$$
||f||_{L^p}= \Big( \int_{{\mathbf{R}}^n}|f(x)|^p dx \Big)^{1/p}
$$
for $1 \leq p < \infty$
and $||f||_{L^{\infty}}= {\rm ess.}\sup_{x \in
{\mathbf{R}}^n}|f(x)|$.
We also define the $L^p$-Sobolev norm  $\| \cdot \|_{L^p_s}$ by 
$$
\| f \|_{L^p_s} = \| ( \langle \cdot \rangle^{s} \widehat{f}(\cdot) )^\vee \|_{L^p}\quad
{\rm with} \quad
\langle \cdot \rangle = (1+|\cdot|^2)^{1/2}.
$$
Let $(X,\| \cdot \|_X)$ and $(Y, \| \cdot \|_Y)$ be two Banach spaces,
which include ${\mathcal S}({\mathbf R}^n)$,
respectively.
We say that an operator $T$ from $X$ to $Y$
is bounded
if there exists a constant $C>0$
such that $\| Tf \|_Y \leq C \| f \|_X$
for all $f \in {\mathcal S}({\mathbf R}^n)$, and we 
set
$$
\| T \|_{X \to Y} = \sup \{ \| Tf \|_Y ~|~ f \in {\mathcal S}({\mathbf R}^n),~ \|f \|_X = 1 \}.
$$
We use the notation $I \lesssim J$ if $I$
is bounded by a constant times $J$,
and we denote $I \approx J$ if
$I \lesssim J$ and $J \lesssim I$.

\subsection{Modulation spaces}
We recall the modulation spaces.
Let $1 \leq p,q \leq  \infty, s \in {\mathbf R} $ and 
$\varphi \in {\mathcal S}({\mathbf R}^n)$
be such that
\begin{equation}
{\rm supp~} \varphi \subset [-1,1]^n
\quad {\rm and~} \quad \sum_{k \in {\mathbf Z}^n} 
\varphi (\xi -k)=1~{\rm for~all~}\xi \in {\mathbf R}^n. \label{window function}
\end{equation}
 Then  the modulation space
$M^{p,q}_s({\mathbf{R}}^n)$ consists of  all tempered
distributions $f
\in {\mathcal{S}}^{\prime}({\mathbf{R}}^n)$ such that the norm
$$
||f||_{M^{p,q}_s} = \Big( \sum_{k \in {\mathbf{Z}}^n}
 \langle k \rangle^{sq} 
 \Big(
\int_{{\mathbf{R}}^n} \big| \varphi(D-k)f (x) \big|^p dx
\Big)^{q/p} \Big)^{1/q} \label{q-n1}
$$
is finite,
with obvious modifications if $p$ or $q= \infty$. 
Here we denote $\varphi(D-k)f(x)= (\varphi(\cdot -k) \widehat{f}(\cdot))^\vee(x)$.

We simply write $M^{p,q}({\mathbf R}^n)$ 
instead of $M^{p,q}_0({\mathbf R}^n)$ when $s=0$.
The space $M^{p,q}_s({\mathbf R}^n)$ is a Banach space which
 is independent of the choice of $\varphi \in {\mathcal S}({\mathbf R}^n)$
satisfying \eqref{window function}
(\cite[Theorem 6.1]{Feichtinger}).
If $1 \leq p,q < \infty$, 
then
${\mathcal S}({\mathbf R}^n)$
is dense in $M^{p,q}_s({\mathbf R}^n)$
(\cite[Theorem 6.1]{Feichtinger}).
If $1 \leq p_1 \leq p_2 \leq \infty$, $1 \leq q_1 \leq q_2 \leq \infty$
and $s_1 \geq s_2$
then $M^{p_1,q_1}_{s_1} ({\mathbf R}^n) \hookrightarrow
M^{p_2,q_2}_{s_2}({\mathbf R}^n)$(\cite[Proposition 6.5]{Feichtinger}).
Let us define by ${\mathcal M}^{p,q}_s({\mathbf R}^n)$
the completion of ${\mathcal S}({\mathbf R}^n)$
under the norm $\| \cdot \|_{M^{p,q}_s}$.
The dual of ${\mathcal M}^{p,q}_s({\mathbf R}^n)$
can be identified with 
${\mathcal M}^{p^\prime,q^\prime}_{-s}({\mathbf R}^n)$,
where
$1/p+1/p^\prime =1 = 1/q + 1/q^\prime$.
If $1 \leq p,q < \infty$, then ${\mathcal M}^{p,q}_s({\mathbf R}^n) = M^{p,q}_s({\mathbf R}^n)$
 (\cite[Lemma 2.2]{Ben-Gro-Hei-Oko}).
Moreover, the complex interpolation theory for these spaces reads as follows:
Let $0< \theta <1$ and $1 \leq  p_1,p_2,q_1, q_2  \leq \infty, s_1,s_2 \in {\mathbf R}$.
Set
$1/p=(1- \theta)/p_1+ \theta/p_2$, $
1/q = (1- \theta)/q_1 + \theta /q_2$ and
$s=(1- \theta)s_1 + \theta s_2$,
then $({\mathcal M}^{p_1,q_1}_{s_1}, {\mathcal M}^{p_2,q_2}_{s_2})_{[\theta]} ={\mathcal M}^{p,q}_s$
(\cite[Theorem 6.1]{Feichtinger},  \cite[Theorem 2.3]{Wang Huang}).

We recall the following lemmas.

\begin{lemma}[{\cite[Proposition 6.7]{Feichtinger}}]\label{inclusion 1}
Let $1 \leq  p \leq \infty$,
$1/p +1/p^\prime =1$ and 
$s \in {\mathbf R}$.
Then 
$$M^{p, \min ( p,p^\prime )}_s ({\mathbf R}^n)
\hookrightarrow
L^p_s ({\mathbf R}^n)
\hookrightarrow 
M^{p, \max ( p,p^\prime )}_s
({\mathbf R}^n).$$  

\end{lemma}

%
%

Let $U_\lambda :f(x) \mapsto f(\lambda x)$
be the dilation operator.
Then the following dilation property of
$M^{p,q}$ is known.

\begin{lemma}[{\cite[Theorem 3.1]{Sugimoto-Tomita}}]\label{DILATION}
Let $1 \leq p,q \leq \infty$.
We have, for $C_1,C_2>0,$
\begin{align*}
\| U_\lambda f \|_{M^{p,q}}
\leq C_1 \lambda^{n \mu_1(p,q)}
\| f \|_{M^{p,q}}, \quad \forall \lambda \geq 1, \forall f \in M^{p,q}({\mathbf R}^n),\\
\| U_\lambda f \|_{M^{p,q}} \geq C_2 \lambda^{n \mu_2(p,q)} \| f \|_{M^{p,q}}, \quad \forall \lambda \geq 1, 
\forall f \in M^{p,q}({\mathbf R}^n),
\end{align*}
where 
\begin{align*}
\mu_1(p,q)=
\begin{cases}
-1/p & {\rm if~} (1/p,1/q) \in I^*_1 ~:~ \min(1/p,1/p^\prime) \geq 1/q, \\
1/q -1 & {\rm if~} (1/p,1/q) \in I^*_2 ~:~  \min(1/q,1/2) \geq 1/p^\prime, \\
-2/p +1/q  & {\rm if~} (1/p,1/q) \in I^*_3  ~:~\min (1/q,1/2) \geq 1/p,
\end{cases}\\
\mu_2(p,q)=
\begin{cases}
-1/p & {\rm if~} (1/p,1/q) \in I_1 ~:~ \max(1/p,1/p^\prime) \leq 1/q, \\
1/q -1 & {\rm if~} (1/p,1/q) \in I_2 ~:~  \max(1/q,1/2) \leq 1/p^\prime, \\
-2/p +1/q  & {\rm if~} (1/p,1/q) \in I_3  ~:~\max (1/q,1/2) \leq 1/p.
\end{cases}
\end{align*}

\end{lemma}

\subsection{Besov spaces}

We recall the Besov spaces.
Let $1 \leq p,q \leq \infty$ and $s \in {\mathbf R}$.
Suppose that $\psi_0,\psi \in {\mathcal S}({\mathbf R}^n)$
satisfy ${\rm supp~}\psi_0 \subset \{ \xi~|~|\xi| \leq 2 \}$,
${\rm supp~} \psi \subset \{ \xi~|~1/2 \leq |\xi| \leq 2 \}$
and $\psi_0(\xi) + \sum^\infty_{j=1} \psi (\xi/2^j)=1$
for all $\xi \in {\mathbf R}^n$.
Set $\psi_j(\cdot) = \psi(\cdot /2^j)$ if $j \geq 1$.
Then the Besov space $B^{p,q}_s({\mathbf R}^n)$
consists of all $f \in {\mathcal S}^\prime({\mathbf R}^n)$
such that
$$
\| f \|_{B^{p,q}_s}
= \left( \sum^\infty_{j=0} 2^{jsq} \| (\widehat{f} (\cdot) \psi_j(\cdot))^\vee \|_{L^p}^q
\right)^{1/q}< \infty,
$$
with usual modification again if $q=\infty$.

The dual of $B^{p,q}_s({\mathbf R}^n)$
can be identified with 
$B^{p^\prime,q^\prime}_{-s}({\mathbf R}^n)$,
where
$1/p+1/p^\prime =1 = 1/q + 1/q^\prime$.

%
%
%
%
%
%

\subsection{Local Hardy spaces}

We  recall the local Hardy spaces.
Let $0<p<\infty$, and let $\Psi \in {\mathcal S}({\mathbf R}^n)$
be such that $\int_{{\mathbf R}^n}\Psi(x)\, dx \neq 0$.
Then the local Hardy space $h^p({\mathbf R}^n)$
consists of all $f \in {\mathcal S}^\prime({\mathbf R}^n)$
such that
$$
\| f  \|_{h^p}= \Big\| \sup_{0<t<1}|\Psi_{t}*f| \Big\|_{L^p}
<\infty,
$$
where $\Psi_t(x)=t^{-n}\Psi(x/t)$.
We remark that
$h^1({\mathbf R}^n) \hookrightarrow L^1({\mathbf R}^n)$
(\cite[Theorem 2]{Goldberg}),
$h^p({\mathbf R}^n)=L^p({\mathbf R}^n)$ if $1<p<\infty$
(\cite[p.30]{Goldberg}),
and the definition of $h^p({\mathbf R}^n)$ is independent
of the choice of $\Psi \in {\mathcal S}({\mathbf R}^n)$
with $\int_{{\mathbf R}^n}\Psi(x)\, dx \neq 0$
(\cite[Theorem 1]{Goldberg}).
The complex interpolation theory
for these spaces reads as follows:
Let $1 \leq p_1,p_2 < \infty$ and
$0< \theta <1$.
Set $1/p=(1- \theta)/p_1 + \theta/p_2$,
then $(h^{p_1},h^{p_2})_{[\theta]} =h^p$ (\cite[p.45]{Triebel II}).

\begin{lemma}[{\cite{Kobayashi Miyachi Tomita}}] \label{LOCAY HARDY h1}
Let $1 \leq q \leq \infty$ and $s \in {\mathbf R}$.
Then $h^1({\mathbf R}^n) \hookrightarrow M^{1,q}_s({\mathbf R}^n)$
if and only if $s \leq -n/q.$
However, in the case $q \not=\infty$,
$L^1({\mathbf R}^n) \hookrightarrow M^{1,q}_s({\mathbf R}^n)$
only if $s< -n/q$.
\end{lemma}

\section{Sufficient conditions}

We prove the {\it if} part of Theorem \ref{MAIN THEOREM}.
First we remark the following fact:
\begin{lemma}\label{S-LEMMA 1}
Let $1<p \leq 2$,
$p \leq q \leq p^\prime$
and $s \leq -n(1/p+1/q-1)$.
Then $L^p({\mathbf R}^n) \hookrightarrow M^{p,q}_s({\mathbf R}^n)$.
\end{lemma}

\begin{proof}[\sc Proof of Lemma \ref{S-LEMMA 1}]
We note that $L^2({\mathbf R}^n)=M^{2,2}({\mathbf R}^n)$
and, by Lemma \ref{LOCAY HARDY h1}, 
$$
h^1({\mathbf R}^n) \hookrightarrow M^{1,q}_{-n/q}({\mathbf R}^n)
$$
for $1 \leq q \leq \infty$.
The complex interpolation method yields
$$
L^p({\mathbf R}^n) \hookrightarrow M^{p,q}_{-n(1/p+1/q-1)}({\mathbf R}^n),
$$
which gives the desired result.
\end{proof}

\begin{proof}[\sc Proof of Theorem \ref{MAIN THEOREM} (``if'' part)]

Suppose $q \leq p$ and $s \leq n \nu_2(p,q)$.
If $q \leq \min(p,p^\prime)$,
then $s \leq n \nu_2(p,q)=0$ and we have
$$
M^{p,q}({\mathbf R}^n) \hookrightarrow 
M^{p,\min(p,p^\prime)}_s ({\mathbf R}^n) \hookrightarrow L^p_s({\mathbf R}^n)
$$
by Lemma \ref{inclusion 1}.
If $2 <p< \infty$ and $p^\prime \leq q \leq p$,
then $s \leq n \nu_2(p,q)=n(1/p+1/q-1)$ and
we have $M^{p,q}({\mathbf R}^n) \hookrightarrow L^p_s({\mathbf R}^n)$ again
by the dual statement of Lemma \ref{S-LEMMA 1}.
Thus we have the sufficiency of conditions (1) and (3).
Conditions (2) and (4) are sufficient by Corollary \ref{LpMpq}.
\end{proof}

\section{Necessary conditions}

We prove the {\it only if} part of Theorem \ref{MAIN THEOREM}.
For the purpose, we prepare lemmas \ref{N-LEMMA 1}--\ref{Pre N-LEMMA 2}
whose proofs are repetitions of arguments
in \cite{Kobayashi Miyachi Tomita}:

\begin{lemma}\label{N-LEMMA 1}
Let $1 \leq p,q \leq \infty$, $p<q$ and $s \in {\mathbf R}$.
If $M^{p,q}_s({\mathbf R}^n) \hookrightarrow L^p({\mathbf R}^n)$,
then
$
s> n(1/p -1/q).
$

\end{lemma}

\begin{lemma}\label{Pre N-LEMMA 1}
Let $1 \leq  p,q \leq \infty$ and $s \in {\mathbf R}$.
If $M^{p,q}_s({\mathbf R}^n) \hookrightarrow L^p({\mathbf R}^n)$,
then 
\[
\|\{c_k\}\|_{\ell^p}
\lesssim \|\{(1+|k|)^{s}\, c_{k}\}\|_{\ell^q}
\]
for all finitely supported sequences $\{c_k\}_{k \in {\mathbf Z}^n}$
$($that is, $c_k=0$ except for a finite number of $k$'s$).$
\end{lemma}

\begin{proof}[\sc Proof of Lemma \ref{Pre N-LEMMA 1}]

Let $\eta \in {\mathcal S} ({\mathbf R}^n) \setminus\{0\}$ be such that
$\mathrm{supp}\, \eta \subset [-1/2,1/2]^n$.
For a finitely supported sequence $\{c_\ell\}_{\ell \in {\mathbf Z}^n}$,
we set
\[
f(x)=\sum_{\ell \in {\mathbf Z}^n}
c_{\ell}\, e^{ i \ell\cdot x}\,
\eta(x-\ell).
\]
Let $\varphi \in {\mathcal S} ({\mathbf R}^n)$ be satisfying
\eqref{window function}.
Since
\[
\widehat{f}(\xi)
=\sum_{\ell \in {\mathbf Z}^n}
c_{\ell}\, e^{ i|\ell|^2}\,
e^{- i \ell\cdot\xi}\, \widehat{\eta}(\xi-\ell),
\]
we see that
\begin{equation}\label{(3.2.1)}
\varphi(D-k)f(x)
= \frac{1}{(2 \pi)^n}
\sum_{\ell \in {\mathbf Z}^n}
c_{\ell}\, e^{ i|\ell|^2}
\int_{{\mathbf R}^n}e^{ i(x-\ell)\cdot\xi}\,
\varphi(\xi-k)\, \widehat{\eta}(\xi-\ell)\, d\xi.
\end{equation}
Using
\[
\int_{{\mathbf R}^n}
(1+|x-y|)^{-M}\, (1+|y|)^{-M}\, dy
\lesssim (1+|x|)^{-M},
\]
where $M>n$, and
\begin{align*}
&(x-\ell)^{\alpha}
\int_{{\mathbf R}^n}e^{ i(x-\ell)\cdot\xi}\,
\varphi(\xi-k)\, \widehat{\eta}(\xi-\ell)\,
d\xi
\\
&=\sum_{\alpha_1+\alpha_2=\alpha}
C_{\alpha_1,\alpha_2}
\int_{{\mathbf R}^n}e^{ i(x-\ell)\cdot\xi}\,
(\partial^{\alpha_1}\varphi)(\xi-k)\,
(\partial^{\alpha_2}
\widehat{\eta})(\xi-\ell)\, d\xi,
\end{align*}
we have
\begin{multline}\label{(3.2.2)}
\left| \int_{{\mathbf R}^n}e^{ i(x-\ell)\cdot\xi}\,
\varphi(\xi-k)\, \widehat{\eta}(\xi-\ell)\,
d\xi \right| 
\le C_N(1+|x-\ell|)^{-N}(1+|k-\ell|)^{-N}
\end{multline}
for all $N \ge 1$.
Let $N$ be 
a sufficiently large integer.
Then, by \eqref{(3.2.1)} and \eqref{(3.2.2)},
\[
|\varphi(D-k)f(x)|
\lesssim \sum_{\ell \in {\mathbf Z}^n}
\frac{|c_{\ell}|}{(1+|x-\ell|)^{N}(1+|k-\ell|)^{N}},
\]
which provides
\begin{align*}
\|\varphi(D-k)f\|_{L^p}
&\lesssim \sum_{\ell \in {\mathbf Z}^n}
\frac{|c_{\ell}|}{(1+|k-\ell|)^{N}}\|(1+|\cdot-\ell|)^{-N}\|_{L^p}
\\
&\approx \sum_{\ell \in {\mathbf Z}^n}
\frac{|c_{\ell}|}{(1+|k-\ell|)^{N}}.
\end{align*}
Then, since
\begin{align*}
\|f\|_{M^{p,q}_s}
&=\left\|\left\{(1+|k|)^{s}\|\varphi(D-k)f\|_{L^p}\right\} \right\|_{\ell^q} \\
& \lesssim
\left\|\left\{(1+|k|)^{s}\sum_{\ell \in {\mathbf Z}^n}
\frac{|c_{\ell}|}{(1+|k-\ell|)^{N}}\right\} \right\|_{\ell^q}
\\
& \lesssim
\left\|\left\{\sum_{\ell \in {\mathbf Z}^n}
\frac{(1+|\ell|)^{s}|c_{\ell}|}{(1+|k-\ell|)^{N-|s|}}\right\} \right\|_{\ell^q},
\end{align*}
we have by Young's inequality
\begin{equation}\label{(3.2.3)}
\|f\|_{M^{p,q}_s} 
\lesssim \| \{ (1+|\ell|)^s c_\ell \} \|_{\ell^q}.
\end{equation}
On the other hand,
since $\mathrm{supp}\, \eta(\cdot-\ell) \subset \ell+[-1/2,1/2]^n$
for all $\ell \in {\mathbf R}^n$,
we see that
\begin{equation}\label{(3.2.5)}
\begin{split}
\|f\|_{L^p}
&=\left( \int_{{\mathbf R}^n}\left| \sum_{\ell \in {\mathbf Z}^n}
c_{\ell}\, e^{ i \ell\cdot x}\,
\eta(x-\ell)\right|^p dx \right)^{1/p}
\\
&=\left( \int_{{\mathbf R}^n}\sum_{\ell \in {\mathbf Z}^n}
\left| c_{\ell}\, e^{ i \ell\cdot x}\,
\eta(x-\ell)\right|^p dx \right)^{1/p}
=\|\eta\|_{L^p}\|\{c_{\ell}\}\|_{\ell^p}
\end{split}
\end{equation}
for $p\not=\infty$.
We have easily the same conclusion for $p=\infty$.
By our assumption
$M^{p,q}_s({\mathbf R}^n) \hookrightarrow L^p ({\mathbf R}^n)$
and \eqref{(3.2.3)}--\eqref{(3.2.5)},
we have
\[
\|\{c_{\ell}\}\|_{\ell^p}
\lesssim \|f\|_{L^p} 
\lesssim \|f\|_{M^{p,q}_s}
\lesssim \|\{(1+|\ell|)^{s}\, c_{\ell}\}\|_{\ell^q}.
\]
The proof is complete.

\end{proof}

\begin{proof}[\sc Proof of Lemma \ref{N-LEMMA 1}]
Suppose 
$M^{p,q}_s ({\mathbf R}^n) \hookrightarrow L^p ({\mathbf R}^n)$.
By Lemma \ref{Pre N-LEMMA 1}, we have
\[
\left( \sum_{k \in {\mathbf Z}^n}|c_{k}|^p\right)^{1/p}
\lesssim \|\{(1+|k|)^{s}\, c_{k}\}\|_{\ell^q}
\]
for all finitely supported sequences $\{c_k\}_{k \in {\mathbf Z}^n}$.
Setting $c_k=(1+|k|)^{-s}\, |d_k|^{1/p}$,
we see that it is equivalent to
\[
\sum_{k \in {\mathbf Z}^n}(1+|k|)^{-sp}\, |d_k|
\lesssim \|\{d_{k}\}\|_{\ell^{q/p}}
\]
for all finitely supported sequences $\{d_k\}_{k \in {\mathbf Z}^n}$.
Hence we have
\[
\|\{(1+|k|)^{-s p}\}\|_{\ell^{(q/p)'}}
=\sup \left| \sum_{k \in {\mathbf Z}^n}(1+|k|)^{-s p}\, d_k\right|
\lesssim 1,
\]
where the supremum is taken over all
finitely supported sequences $\{d_k\}_{k \in {\mathbf Z}^n}$
such that $\|\{d_k\}\|_{\ell^{q/p}}=1$.
Note here that $(q/p)'<\infty$ from the assumption $p<q$.
Hence $p,q,s$ must satisfy $sp(q/p)'>n$,
that is, $s>n(1/p-1/q)$.

\end{proof}

\begin{lemma}\label{N-LEMMA 2}
Let $1 \leq  q<p<\infty$
and $s \in {\mathbf R}$.
If $L^p({\mathbf R}^n) \hookrightarrow M^{p,q}_s({\mathbf R}^n)$,
then $s<-n (1/p +1/q -1)$.

\end{lemma}

\begin{lemma}\label{Pre N-LEMMA 2}
Let $1 \leq p,q <\infty$ and $s \in {\mathbf R}$.
If $L^p({\mathbf R}^n) \hookrightarrow M^{p,q}_s({\mathbf R}^n)$,
then 
$$
\left\{ \sum_{k \neq 0}|k|^{(n(1/p-1)+s)q}
\left( \sum_{|k|/2 \leq |\ell| \le 2|k|}
|c_\ell|^p \right)^{q/p} \right\}^{1/q}
\lesssim \left( \sum_{k \neq 0}
|c_{k}|^p\right)^{1/p}
$$
for all finitely supported sequences
$\{c_k\}_{k \in {\mathbf Z}^n\setminus\{0\}}$.

\end{lemma}

\begin{proof}[\sc Proof of Lemma \ref{Pre N-LEMMA 2}]
Let $0< \delta <1$ and $a \in {\mathcal S}({\mathbf R}^n)$
be such that
\begin{align*}
{\rm supp~}a \subset [- \delta/8,\delta/8]^n, ~~ 
\| a \|_{L^\infty} \leq 1,~~{\rm and}~~
|\widehat{a}(\xi)| \geq C>0 ~{\rm on}~ |\xi| \leq 2
\end{align*}
(see, for example,  \cite[Lemma 4.3]{Kobayashi Miyachi Tomita}).
For a finitely supported sequence $\{ c_n \}_{\ell \in {\mathbf Z}^n \setminus \{ 0 \}}$,
we define $f \in {\mathcal S}({\mathbf R}^n)$ by
$$
f(x) = \sum_{\ell \not=0} c_\ell |\ell|^{n/p} a(|\ell|(x- \ell)).
$$
We first estimate $\| f \|_{L^p}$.
Since
$$
{\rm supp~}a(|\ell|(\cdot - \ell)) \subset \ell + [- \delta/(8|\ell|), \delta/(8|\ell|)]^n,
$$
we have
\begin{align*}
\| f \|^p_{L^p}
&= \int_{{\mathbf R}^n} \Big| \sum_{\ell \not= 0}
c_\ell |\ell |^{n/p} a(|\ell|(x- \ell)) \Big|^p dx \\
&= \int_{{\mathbf R}^n} \sum_{\ell \not= 0} |c_\ell|^p |\ell|^n 
|a(|\ell| (x-\ell))|^p dx 
= \| a \|^p_{L^p} \sum_{\ell \not= 0} |c_\ell|^p.
\end{align*}
Next, we estimate $\| f \|_{M^{p,q}_s}$.
We note the following facts:
\begin{enumerate}
\item[\bf Fact 1.] Let $\Psi \in {\mathcal S}({\mathbf R}^n)$ be such that
$
\Psi =1$ on $[-\delta/4, \delta/4]^n$, 
${\rm supp~} \Psi \subset [-3 \delta/8, 3 \delta/8]^n,$ and
$|\widehat{\Psi}| \geq C>0$ on $[-2,2]^n.
$
Then we have
$$
\| f \|_{M^{p,q}_s} \approx 
\left( \sum_{k \in {\mathbf Z}^n}  (1+|k|)^{sq} \| f*(M_k \Psi) \|^q_{L^p} \right)^{1/q},
$$
where $M_k \Psi (x)= e^{ik \cdot x} \Psi(x)$.

\item[\bf Fact 2.]
For all $\ell \not=0 $, we have
$$
{\rm supp~} a(|\ell| (\cdot - \ell))
\subset \ell + [- \delta /8|\ell|, \delta / 8 |\ell|]^n
\subset \ell + [-\delta/8, \delta/8]^n.
$$

\item[\bf Fact 3.]
For all $ x \in m +[- \delta/8, \delta/8]^n, m \in {\mathbf Z}^n$, we have
$$
{\rm supp~} \Psi(x- \cdot)
\subset x+ [-3 \delta /8, 3 \delta/8]^n
\subset m+[-\delta/2,\delta/2]^n.
$$

\end{enumerate}
From these facts, we have
\begin{align*}
&\|(M_k\Psi)*f\|_{L^p}^p \\
&\ge \sum_{m \in {\mathbf Z}^n}\int_{m+[- \frac{\delta}{8},\frac{\delta}{8}]^n}
|(M_k\Psi)*f(x)|^p\, dx
\\
&=\sum_{m \in {\mathbf Z}^n}\int_{m+[- \frac{\delta}{8},\frac{\delta}{8}]^n}
\left| \int_{{\mathbf R}^n}e^{i k\cdot(x-y)}\,
\Psi(x-y)\, \sum_{\ell \neq 0}
c_\ell\, |\ell|^{\frac{n}{p}}\, a(|\ell|(y-\ell))\,
dy \right|^p dx
\\
&=\sum_{m \neq 0}\int_{m+[- \frac{\delta}{8},\frac{\delta}{8}]^n}
\left| \int_{{\mathbf R}^n}e^{- i k\cdot y}\,
\Psi(x-y)\, c_m\, |m|^{\frac{n}{p}}\,
a(|m|(y-m))\, dy \right|^p dx.
\end{align*}
If $x \in m+[-\delta/8,\delta/8]^n$ and 
$y \in \mathrm{supp}\, a(|m|(\cdot-m))$,
then 
$$x-y \in (m+[-\delta/8,\delta/8]^n)
-(m+[-\delta/8,\delta/8]^n)=[-\delta/4,\delta/4]^n,$$
and so $\Psi(x-y)=1$.
Hence,
\begin{align*}
&\sum_{m \neq 0}\int_{m+[- \frac{\delta}{8},\frac{\delta}{8}]^n}
\left| \int_{{\mathbf R}^n}e^{- i k\cdot y}\,
\Psi(x-y)\, c_m\, |m|^{\frac{n}{p}}\,
a(|m|(y-m))\, dy \right|^p dx
\\
&=\sum_{m \neq 0}\int_{m+[- \frac{\delta}{8},\frac{\delta}{8}]^n}
\left| \int_{{\mathbf R}^n}e^{- i k\cdot y}\,
c_m\, |m|^{\frac{n}{p}}\,
a(|m|(y-m))\, dy \right|^p dx
\\
&=\sum_{m \neq 0}\int_{m+[- \frac{\delta}{8},\frac{\delta}{8}]^n}
\left| c_m\, |m|^{\frac{n}{p}}\, |m|^{-n}\,
\widehat{a} \left(\frac{k}{|m|}\right) \right|^p dx
\\
&=\left( \frac{\delta}{4} \right)^{n}\sum_{m \neq 0}
|c_m|^p\, |m|^{n-pn}\,
\left| \widehat{a} \left( \frac{k}{|m|} \right) \right|^p.
\end{align*}
Moreover, using $|\widehat{a}(\xi)| \geq C>0$ 
for all  $1/2 \leq |\xi| \leq 2$,
we obtain 
\begin{align*}
\|(M_k\Psi)*f\|_{L^p}^p 
&\ge (\delta/4)^{n}\sum_{m \neq 0}
|c_m|^p\, |m|^{n-pn}\,
|\widehat{a}(k/|m|)|^p
\\
&\ge (\delta/4)^{n}\sum_{|k|/2 \leq |m| \le 2|k|}
|c_m|^p\, |m|^{n-pn}\,
|\widehat{a}(k/|m|)|^p
\\
&\gtrsim \sum_{|k|/2 \leq  |m| \le 2|k|}
|c_m|^p\, |m|^{n-pn}
\gtrsim |k|^{n-pn}
\sum_{|k|/2 \leq |m| \le 2|k|}|c_m|^p
\end{align*}
for all $k \not= 0$. Then
\begin{equation*}
\begin{split}
\|f\|_{M^{p,q}_s}
&\approx \left( \sum_{k \in {\mathbf Z}^n}(1+|k|)^{sq}
\|(M_k\Psi)*f\|_{L^p}^q \right)^{1/q}
\\
&\gtrsim \left\{ \sum_{k \neq 0}(1+|k|)^{sq}
\left( |k|^{n-pn} \sum_{|k|/2 \leq |m| \le 2|k|}
|c_m|^p \right)^{q/p} \right\}^{1/q}
\\
&\gtrsim \left\{ \sum_{k \neq 0}|k|^{(n(1/p-1)+s)q}
\left( \sum_{|k|/2 \leq |m| \le 2|k|}
|c_m|^p \right)^{q/p} \right\}^{1/q}.
\end{split}
\end{equation*}
Therefore, by our assumption 
$L^p({\mathbf R}^n) \hookrightarrow M^{p,q}_s({\mathbf R}^n)$,
we have
\begin{multline*}
\left\{ \sum_{k \neq 0}|k|^{(n(1/p-1)+s)q}
\left( \sum_{|k|/2 \leq |m| \le 2|k|}
|c_m|^p \right)^{q/p} \right\}^{1/q} \\
\lesssim \|f\|_{M^{p,q}_s} 
\lesssim \|f\|_{L^p} 
 \lesssim \left(\sum_{\ell \neq 0}
|c_{\ell}|^p \right)^{1/p}.
\end{multline*}

\end{proof}

\begin{proof}[\sc Proof of Lemma \ref{N-LEMMA 2}]
Suppose that $s \ge -n(1/p+1/q-1)$ contrary to our claim.
Noting that $q/p<1$ from the assumption $q<p$, take $\varepsilon >0$ such that
$(1+\varepsilon)q/p<1$
and define 
$\{c_k\}_{k \in {\mathbf Z}^n\setminus\{0\}}$
by
$$
c_k=
\begin{cases}
|k|^{-n/p}\left(\log |k|\right)^{-(1+\varepsilon)/p}
&if \quad |k| \ge N, \\
0 &if \quad |k|<N,
\end{cases}
$$
where $N$ is a sufficiently large.
Note also that
$\{|k|^{-n/r}(\log |k|)^{-\alpha/r}\}_{|k| \ge N}
\in \ell^r$ if $\alpha>1$,
and $\{|k|^{-n/r}(\log |k|)^{-\alpha/r}\}_{|k| \ge N}
\notin \ell^r$ if $\alpha \le 1$,
where $r<\infty$ 
(see, for example, \cite[Remark 4.3]{Sugimoto-Tomita-2}).
Thus 
\begin{equation*}
\left(\sum_{k \neq 0}|c_k|^p \right)^{1/p}
=\left\{\sum_{|k| \ge N}
\left( |k|^{-n/p}\left(\log |k|\right)^{-(1+\varepsilon)/p}
\right)^p\right\}^{1/p}<\infty.
\end{equation*}
On the other hand,
\begin{equation*}
\left\{ \sum_{k \neq 0}|k|^{(n(1/p-1)+s)q}
\left( \sum_{|k|/2 \le |\ell| \le 2|k|}
|c_\ell|^p \right)^{q/p} \right\}^{1/q}=\infty.
\end{equation*}
In fact, since $n(1/p-1)+s \ge -n/q$
and $(1+\varepsilon)q/p<1$, we see that
\begin{align*}
&\left\{ \sum_{k \neq 0}|k|^{(n(1/p-1)+s)q}
\left( \sum_{|k|/2 \le |\ell| \le 2|k|}
|c_\ell|^p \right)^{q/p} \right\}^{1/q}
\\
&\ge \left\{ \sum_{|k| \ge 2N}|k|^{(n(1/p-1)+s)q}
\left( \sum_{|k|/2 \le |\ell| \le 2|k|}
\left( |\ell|^{-n/p}\, (\log |\ell|)^{-(1+\varepsilon)/p}
\right)^p \right)^{q/p} \right\}^{1/q}
\\
&\gtrsim \left\{ \sum_{|k| \ge 2N}|k|^{(n(1/p-1)+s)q}
(\log|k|)^{-(1+\varepsilon)q/p}\right\}^{1/q}
\\
&\gtrsim \left\{ \sum_{|k| \ge 2N}
\left(|k|^{-n/q}
(\log|k|)^{-\{(1+\varepsilon)q/p\}/q}
\right)^q\right\}^{1/q}=\infty.
\end{align*}
However, this contradicts
Lemma \ref{Pre N-LEMMA 2}.

\end{proof}

\begin{proof}[\sc Proof of Theorem \ref{MAIN THEOREM} (\lq\lq only if" part)]

%

Suppose $M^{p,q}({\mathbf R}^n) \hookrightarrow L^p_s({\mathbf R}^n)$.
Then we have $s \leq  n \nu_2(p,q)$ by Corollary \ref{LpMpq}.
Particularly in the case $p<q$, we have
$s< -n(1/p-1/q)=n\nu_2(p,q)$
for $p\leq2$ by Lemma \ref{N-LEMMA 1},
and
$
s< -n(1/p^\prime+ 1/q^\prime -1)
= n(1/p+1/q-1)=n\nu_2(p,q)
$
for $2\leq p$ by the dual statement of Lemma \ref{N-LEMMA 2}.
In the case $p=\infty$,
since ${\mathcal M}^{\infty,q}({\mathbf R}^n)
\hookrightarrow M^{\infty,q}({\mathbf R}^n)
\hookrightarrow L^{\infty}_s({\mathbf R}^n)
$,
we have
$$
L^1_{-s}({\mathbf R}^n) \hookrightarrow ((L^1_{-s}({\mathbf R}^n))^*)^*
= (L^\infty_s({\mathbf R}^n))^* \hookrightarrow
({\mathcal M}^{\infty,q}({\mathbf R}^n))^*
= M^{1,q^\prime}({\mathbf R}^n),
$$
which means that
 $L^1({\mathbf R}^n) \hookrightarrow M^{1,q^\prime}_{s}({\mathbf R}^n)$.
Then we have
$
s <- n/q^\prime =n(1/q-1)
=n \nu_2(\infty,q)
$
for $q\not=1$ by Lemma \ref{LOCAY HARDY h1}.
All of these results yields the necessity of conditions (1)--(4).
\end{proof}

\section{Applications}

\subsection{Unimodular Fourier multiplier}

We consider the
unimodular Fourier multiplier
$e^{i|D|^\alpha}, \alpha \geq 0$,
defined by
\begin{equation*}
e^{i|D|^\alpha} f (x)= \int_{{\mathbf R}^n}
e^{i|\xi|^\alpha} \widehat{f}(\xi) e^{ix \cdot \xi} d \xi, \quad f \in {\mathcal S}({\mathbf R}^n).
\label{DEF OF FOURIER MULTIPLIER}
\end{equation*}
The operator
$e^{i|D|^\alpha} $
has an intimate connection with
the solution $u(t,x)$ of initial value problem
for the dispersive equation
\begin{equation*}
\begin{cases}
i \partial_t u + |\Delta|^{\alpha/2} u =0, \\
u(0,x) =f(x),
\end{cases}
\end{equation*}
$(t,x) \in {\mathbf R} \times {\mathbf R}^n$.
The boundedness of $e^{i|D|^\alpha}$
on several function spaces
has been studied extensively by many authors.
Concerning the $L^p$-Sobolev spaces $L^p_s$
and the modulation spaces $M^{p,q}_s$,
the following theorems
are known.

\begin{previousthmA}[{Miyachi \cite{Miyachi 1}}]
Let  $1 < p < \infty,s \in {\mathbf R}$ and $\alpha >1$.
Then $e^{i|D|^\alpha}$ is bounded from
$L^p_s({\mathbf R}^n)$ to $L^p({\mathbf R}^n)$
if and only if $s \geq \alpha n |1/p -1/2|$.
\end{previousthmA}

\begin{previousthmB}[{B\'{e}nyi-Gr\"{o}chenig-Okoudjou-Rogers
\cite{Benyi et all}}]
Let $1 \leq p,q \leq \infty,s \in {\mathbf R}$ and $0 \leq \alpha \leq  2$.
Then $e^{i|D|^\alpha}$ is bounded
from $M^{p,q}({\mathbf R}^n)$ to $M^{p,q}({\mathbf R}^n)$.
\end{previousthmB}

\begin{previousthmC}[{Miyachi-Nicola-Rivetti-Tabacco-Tomita
\cite{Miyachi-NicolaRivetti-Tabacco-Tomita}}]
Let $1 \leq p,q \leq \infty,s \in {\mathbf R}$ and $\alpha >2$.
Then $e^{i|D|^\alpha}$ is bounded
from $M^{p,q}_s({\mathbf R}^n)$ to $M^{p,q}({\mathbf R}^n)$
if and only if $s \geq (\alpha-2)n|1/p -1/2|$.
\end{previousthmC}

Theorem A says that the operator $e^{i|D|^\alpha}$ is not bounded on $L^p({\mathbf R}^n)$,
and we have generally a loss of regularity of the order 
up to $\alpha n |1/p -1/2|$.
Theorem B describes an advantage of modulation spaces because
we have no loss in the case $0\leq\alpha\leq2$
or smaller loss in the case $\alpha>2$
if we consider the operator $e^{i|D|^\alpha}$ on these spaces.

Then what is the exact order of the loss
when we consider the operator $e^{i|D|^\alpha}$ 
between $L^p$ spaces and modulation spaces.
We can answer this question by using our main theorem.
The case $0 \leq \alpha \leq 2$ is rather simple, 
and we have the following results:

\begin{theorem}\label{THEOREM 0<a<2}
Let $1\leq p,q\leq\infty$, $s\in {\mathbf R}$ and $0\leq\alpha\leq2$.
Then $e^{i|D|^\alpha}$ is bounded from $M^{p,q}_s({\mathbf R}^n)$
to $L^p({\mathbf R}^n)$
if and only if
$M^{p,q}_s({\mathbf R}^n)\hookrightarrow L^p({\mathbf R}^n)$.
\end{theorem}
\begin{proof}
Assume that $e^{i|D|^\alpha}$ is bounded from $M^{p,q}_s({\mathbf R}^n)$ to $L^p({\mathbf R}^n)$:
\[
\|{e^{i|D|^\alpha}f}\|_{L^p}\lesssim \|{f}\|_{M^{p,q}_s}.
\]
Note that $e^{-i|D|^\alpha}$ is
also bounded from $M^{p,q}_s({\mathbf R}^n)$ to $L^p({\mathbf R}^n)$.
Then by taking $f=e^{-i|D|^\alpha}g$ we have
\[
\|{g}\|_{L^p}\lesssim \|{e^{-i|D|^\alpha}g}\|_{M^{p,q}_s}
\lesssim \|{g}\|_{M^{p,q}_s},
\]
which means that $M^{p,q}_s ({\mathbf R}^n) \hookrightarrow L^p({\mathbf R}^n)$
by the equation $\overline{e^{i|D|^\alpha}f}=e^{-i|D|^\alpha}\overline{f}$.
Conversely assume that $M^{p,q}_s ({\mathbf R}^n)  \hookrightarrow L^p({\mathbf R}^n)$.
Since $e^{i|D|^\alpha}$ is bounded on $M^{p,q}_s({\mathbf R}^n)$
by Theorem B,
we have
\[
\|{e^{i|D|^\alpha}f}\|_{L^p}
\lesssim 
\|{e^{i|D|^\alpha}f}\|_{M^{p,q}_s}
\lesssim 
\|{f}\|_{M^{p,q}_s},
\]
which means that $e^{i|D|^\alpha}$ is bounded from $M^{p,q}_s({\mathbf R}^n)$ to $L^p({\mathbf R}^n)$.
\end{proof}

The following corollary is straightforwardly obtained
from Theorem \ref{THEOREM 0<a<2} and Theorem \ref{MAIN THEOREM}.
The second part is just the dual statement of the first part:

\begin{corollary}\label{cor5.2}
Let $1 \leq p,q \leq \infty,s \in {\mathbf R}$ and $0 \leq \alpha \leq  2$.
Then $e^{i|D|^\alpha}$
is bounded from $M^{p,q}_s({\mathbf R}^n)$
to $L^p({\mathbf R}^n)$
if and only if one of the following conditions is satisfied$:$
\begin{enumerate}
\item[$(1)$]  $q \leq p< \infty$ and $s \geq -n \nu_2(p,q);$

\item[$(2)$]  $p<q$ and $s> -n \nu_2(p,q);$
 
\item[$(3)$]  $p=\infty, q=1$, and $s \geq -n \nu_2(\infty,1);$

\item[$(4)$]  $p=\infty, q \not=1$, and $s > -n \nu_2(\infty,q)$,
\end{enumerate}
and from $L^p_s({\mathbf R}^n)$
to $M^{p,q}({\mathbf R}^n)$
if and only if one of the following conditions is satisfied$:$
\begin{enumerate}
\item[$(5)$] $q \geq  p> 1$ and $s \geq n \nu_1(p,q);$

\item[$(6)$] $p>q$ and $s> n \nu_1(p,q);$

\item[$(7)$] $p=1, q=\infty$, and $s \geq n \nu_1(1,\infty);$

\item[$(8)$] $p=1, q \not=\infty$ and $s> n \nu_1(1,q)$.

\end{enumerate}
\end{corollary}

For $\alpha > 2$, 
we have the following results:

\begin{theorem}\label{THEOREM 2<a<infty}
Let $1\leq p,q\leq\infty$, $s\in{\mathbf R}$ and $\alpha >2$.
Then $e^{i|D|^\alpha}$ is bounded from $M^{p,q}_{s+(\alpha-2)n|1/p-1/2|}({\mathbf R}^n)$
to $L^p({\mathbf R}^n)$ if $M^{p,q}_s({\mathbf R}^n)\hookrightarrow L^p({\mathbf R}^n)$.
\end{theorem}
\begin{proof}
Assume that $M^{p,q}_s ({\mathbf R}^n) \hookrightarrow L^p({\mathbf R}^n)$.
Since $e^{i|D|^\alpha}$ is bounded from $M^{p,q}_{s+(\alpha-2)n|1/p-1/2|}({\mathbf R}^n)$
to $M^{p,q}_s({\mathbf R}^n)$
by Theorem C,
we have
\[
\|{e^{i|D|^\alpha}f}\|_{L^p}
\lesssim 
\|{e^{i|D|^\alpha}f}\|_{M^{p,q}_s}
\lesssim 
\|{f}\|_{M^{p,q}_{s+(\alpha-2)n|1/p-1/2|}},
\]
which means that $e^{i|D|^\alpha}$ is bounded from
$M^{p,q}_{s+(\alpha-2)n|1/p-1/2|}({\mathbf R}^n)$ to $L^p({\mathbf R}^n)$.
\end{proof}

The following corollary is obtained
from Theorem \ref{THEOREM 2<a<infty}, Theorem \ref{MAIN THEOREM} 
and the duality argument again:

\begin{corollary}\label{MAIN THEOREM 4}
Let $1 \leq p,q \leq \infty,s \in {\mathbf R}$ and  $\alpha >2$.
Then $e^{i|D|^\alpha}$
is bounded from $M^{p,q}_s({\mathbf R}^n)$
to $L^p({\mathbf R}^n)$
if one of the following conditions is satisfied$:$
\begin{enumerate}
\item[$(1)$]  $q \leq p< \infty$ and $s \geq -n \nu_2(p,q)+(\alpha-2)n|1/p -1/2|;$

\item[$(2)$]  $p<q$ and $s> -n \nu_2(p,q)+(\alpha-2)n|1/p -1/2|;$
 
\item[$(3)$]  $p=\infty, q=1$, and $s \geq -n \nu_2(\infty,1)+(\alpha-2)n|1/p -1/2|;$

\item[$(4)$]  $p=\infty, q \not=1$, and $s> -n \nu_2(\infty,q)+(\alpha-2)n|1/p -1/2|$,

\end{enumerate}
and from $L^p_s({\mathbf R}^n)$
to $M^{p,q}({\mathbf R}^n)$
if one of the following conditions is satisfied$:$
\begin{enumerate}
\item[$(5)$] $q \geq  p> 1$ and $s \geq n \nu_1(p,q)+(\alpha-2)n|1/p -1/2|$;

\item[$(6)$] $p>q$ and $s> n \nu_1(p,q)+(\alpha-2)n|1/p -1/2|;$

\item[$(7)$] $p=1, q=\infty$, and $s \geq n \nu_1(1,\infty)+(\alpha-2)n|1/p -1/2|;$

\item[$(8)$] $p=1, q \not=\infty$ and $s> n \nu_1(1,q)+(\alpha-2)n|1/p -1/2|$.

\end{enumerate}
\end{corollary}

Although the converse of Theorem \ref{THEOREM 2<a<infty} is not true,
we believe that the converse of Corollary \ref{MAIN THEOREM 4}  is still true.
In fact we have at least the following result:

\begin{theorem}\label{MAIN THEOREM 5}
Let $1 \leq p,q \leq \infty,s \in {\mathbf R}$ and  $\alpha >2$.
Suppose that $e^{i|D|^\alpha}$
is bounded from $M^{p,q}_s({\mathbf R}^n)$
to $L^p({\mathbf R}^n)$.
Then we have $s \geq -n \nu_2(p,q)+(\alpha-2)n|1/p -1/2|$.
Suppose that $e^{i|D|^\alpha}$
is bounded from $L^p_s({\mathbf R}^n)$
to $M^{p,q}({\mathbf R}^n)$ instead.
Then we have $s \geq n \nu_1(p,q)+(\alpha-2)n|1/p -1/2|$.

\end{theorem}

To prove Theorem \ref{MAIN THEOREM 5}, we use the following lemma.

\begin{lemma}\label{MULTIPLIER LEMMA 0}
Let $1 \leq  p,q \leq \infty,s \in {\mathbf R}$ and
$\alpha \geq 0$.
Then
\begin{equation}
 \sup_{k \in {\mathbf Z}^n}
\langle k \rangle^{-s}
\| \varphi(D-k) e^{i|D|^\alpha} \|_{L^p \to L^p}
\lesssim
\| e^{i|D|^\alpha} \|_{M^{p,q}_s \to L^p},
\end{equation}
where $\varphi$ is a function satisfying \eqref{window function}.
\end{lemma}
\begin{proof}

Let  $N$ be a positive integer
such that
\begin{equation*}
\varphi(\cdot - k)
= \sum_{|\ell| \leq N} \varphi(\cdot-k) \varphi(\cdot - (k+\ell))
\end{equation*}
for all $k \in {\mathbf Z}^n.$
Then
we have
\begin{align*}
&\| \varphi(D-k)e^{i|D|^\alpha} f \|_{L^p}  \\
&\leq \| e^{i|D|^\alpha} \|_{M^{p,q}_s \to L^p} \| \varphi(D-k) f \|_{M^{p,q}_s}\\
&= \| e^{i|D|^\alpha} \|_{M^{p,q}_s \to L^p}
\Big( \sum_{m \in {\mathbf Z}^n}  \langle m \rangle^{sq} 
\| \varphi(D- m)
\varphi(D-k) f \|_{L^p}^q \Big)^{1/q}
\\
&= \| e^{i|D|^\alpha} \|_{M^{p,q}_s \to L^p}
\Big( \sum_{|\ell| \leq N}  \langle k+\ell \rangle^{sq} 
\| \varphi(D- (k+\ell))
\varphi(D-k)f \|_{L^p}^q \Big)^{1/q}
\\
& \lesssim  \langle k \rangle^s
\| e^{i|D|^\alpha} \|_{M^{p,q}_s \to L^p}
\| f \|_{L^p}
\end{align*}
for $f \in {\mathcal S}({\mathbf R}^n)$ and $q \not= \infty$.
We have easily the same conclusion for $q= \infty$.
Hence, we obtain the desired result.

\end{proof}

\begin{remark}\label{equivalence}
We remark that
since
\begin{align*}
\| e^{i|D|^\alpha} \|_{M^{p,\widetilde{q}}_s \to M^{p,\widetilde{q}}}
\approx \sup_{k \in {\mathbf Z}^n}
\langle k \rangle^{-s}
\| \varphi(D-k) e^{i|D|^\alpha} \|_{L^p \to L^p} 
\end{align*}
for  $1 \leq p, \widetilde{q} \leq \infty,s \in {\mathbf R}  $
(see \cite[Lemma 2.2]{Miyachi-NicolaRivetti-Tabacco-Tomita}),
we have
\begin{equation*}
\| e^{i|D|^\alpha} \|_{M^{p,\tilde{q}}_s \to M^{p,\tilde{q}}}
\lesssim
\| e^{i|D|^\alpha} \|_{M^{p,q}_s \to L^p}
\end{equation*}
for all $1 \leq  p, q, \tilde{q} \leq \infty.$
\end{remark}

Now, we prove Theorem \ref{MAIN THEOREM 5}.

\begin{proof}[Proof of Theorem \ref{MAIN THEOREM 5}]
Since the latter
is just the dual statement of the former,
we prove only the former.
Suppose that $e^{i|D|^\alpha}$
is bounded from $M^{p,q}_s({\mathbf R}^n)$
to $L^p({\mathbf R}^n)$:
\begin{equation}
\| e^{i|D|^\alpha} f \|_{L^p} \lesssim \| f \|_{M^{p,q}_s}, 
\quad f \in {\mathcal S}({\mathbf R}^n). \label{main estimate}
\end{equation}
\noindent
(i) Let $q \leq \min(p,p^\prime)$.
By the necessary condition of Theorem C and Remark \ref{equivalence},
we have
$s \geq (\alpha -2)n|1/p-1/2|$.
Since $\nu_2(p,q)=0$, we obtain the desired result.

\noindent
(ii)
Let $1 \leq p \leq 2$ and $p \leq q  \leq p^\prime$.
Note that  inequality
 \eqref{main estimate} 
can be written as
\begin{equation}\label{ASSUME 2}
\| e^{i|D|^\alpha} \langle D \rangle^{-s} f \|_{L^p}
\lesssim \| f \|_{M^{p,q}}, \quad f \in {\mathcal S} ({\mathbf R}^n)
\end{equation}
by the lifting property.
Here, we denote
$\langle D \rangle^{-s} f = (\langle \cdot \rangle^{-s} \widehat{f}(\cdot))^\vee$
for  $s \in {\mathbf R}$.

Let $g \in {\mathcal S}({\mathbf R}^n)$ be
such that
\begin{equation}
{\rm supp~} \widehat{g} \subset \{\xi~|~  2^{-1} <|\xi| < 2\} 
~~~ {\rm and} ~~~ \widehat{g}(\xi)=1 ~{\rm on}~ \{ \xi~|~  2^{-1/2} <|\xi| < 2^{1/2} \},  \label{CONDITION g}
\end{equation}
and test  \eqref{ASSUME 2} with a specific $f= U_\lambda g$,
$\lambda \geq 1$.
Since 
$$e^{i|D|^\alpha} \langle D \rangle^{-s} U_\lambda g
= U_\lambda (e^{i| \lambda D|^\alpha} \langle \lambda D \rangle^{-s}g),$$
it follows 
from Theorem \ref{DILATION} that
$$
\lambda^{-n/p} \| e^{i|\lambda D|^\alpha} \langle \lambda D \rangle^{-s} g \|_{L^p}
\lesssim \lambda^{n \mu_1(p,q)} \| g \|_{M^{p,q}}.
$$

On the other hand, by the change of variable
$x \mapsto \lambda^\alpha x$ and the
method of stationary phase,
we obtain
\begin{align*}
 \| e^{i|\lambda D|^\alpha} \langle \lambda D \rangle^{-s} g \|_{L^p}
&= \left\| \int_{{\mathbf R}^n} e^{ix \cdot \xi + i |\lambda \xi|^\alpha}
\langle \lambda \xi \rangle^{-s} \widehat{g} (\xi) d  \xi \right\|_{L^p} \\
&= \lambda^{\alpha n/p} \left\| \int_{{\mathbf R}^n} e^{i \lambda^\alpha (x \cdot \xi +  |\xi|^\alpha)}
\langle \lambda \xi \rangle^{-s} \widehat{g} (\xi) d  \xi \right\|_{L^p} \\
& \gtrsim \lambda^{\alpha n/p - \alpha n/2 -s}.
\end{align*}
Combining these two estimates, we get
$$
\lambda^{n \mu_1(p,q) +n/p - \alpha n/p +\alpha n/2 +s}
\gtrsim 1
$$
for all $\lambda \geq 1$.
Letting $\lambda  \to \infty$ yields the 
necessary condition
\begin{align*}
s &\geq -n(1/q-1)
- n/p +\alpha n/p - \alpha n/2 \\
&=(\alpha -2)n(1/p-1/2)
+ n/p-n/q \\
&= (\alpha -2)n|1/p-1/2| -n \nu_2(p,q),
\end{align*}
since $\mu_1(p,q)=1/q-1$ and $\nu_2(p,q)= -1/p+1/q$.

\noindent
(iii)
Let $2 \leq p \leq \infty$ and $p^\prime \leq q  \leq p$.
By duality, we have
$$
\| e^{-i|D|^\alpha} f \|_{L^{p^\prime}}
\lesssim \| f \|_{M^{p^\prime,q^\prime}_{-s}},
\quad f \in {\mathcal S}({\mathbf R}^n).
$$
So, we have only to prove
the following lemma.

\begin{lemma}\label{Lps-Mpq estimate}
Let $1 \leq p^\prime \leq 2, p^\prime \leq q^\prime \leq p$
and $s \in {\mathbf R}$.
If $e^{i|D|^\alpha}$ is bounded from
$L^{p^\prime}_s({\mathbf R}^n)$ to $M^{p^\prime,q^\prime}({\mathbf R}^n)$,
then $s \geq (\alpha -2) n|1/p-1/2| -n \nu_2(p,q)$.

\end{lemma}

\begin{proof}[Proof of Lemma \ref{Lps-Mpq estimate}]
Set $f= U_\lambda g, \lambda \geq 1$,
where $g$ is a function satisfying
\eqref{CONDITION g}.
Then, by Lemma \ref{DILATION},
we have
\begin{align*}
\| e^{i|D|^\alpha} f \|_{M^{p^\prime,q^\prime}}
&= \| e^{i|D|^\alpha} U_\lambda g \|_{M^{p^\prime,q^\prime}} \\
&= \| U_\lambda(e^{i|\lambda D|^\alpha} g) \|_{M^{p^\prime,q^\prime}}\\
&\gtrsim \lambda^{n \mu_2(p^\prime,q^\prime)} 
\| e^{i|\lambda D|^\alpha} g \|_{M^{p^\prime,q^\prime}}.
\end{align*}
In the same way as (ii), we obtain, by the change of variable
$x \mapsto \lambda^\alpha x$ and the
method of stationary phase, 
\begin{align*}
 \| e^{i|\lambda D|^\alpha}  g \|_{M^{p^\prime,q^\prime}} 
 &= \left( \sum_{k \in {\mathbf Z}^n} \| \varphi(D- k) 
e^{i|\lambda D|^\alpha}  g \|_{L^{p^\prime}}^{q^\prime} \right)^{1/q^\prime} \\
& \gtrsim \left\| \int_{{\mathbf R}^n} \varphi(\xi) e^{i x \cdot \xi + i |\lambda \xi|^\alpha}
 \widehat{g} (\xi) d  \xi \right\|_{L^{p^\prime}} \\
&= \lambda^{\alpha n/p^\prime} \left\| \int_{{\mathbf R}^n} e^{i \lambda^\alpha (x \cdot \xi +  |\xi|^\alpha)}
\varphi(\xi)  \widehat{g} (\xi) d  \xi \right\|_{L^{p^\prime}} \\
& \gtrsim \lambda^{\alpha n/p^\prime - \alpha n/2}.
\end{align*}
Hence, we have
$$
\| e^{i|D|^\alpha} f \|_{M^{p^\prime,q^\prime}} \gtrsim  \lambda^{n \mu_2(p^\prime,q^\prime)}   \lambda^{\alpha n(1/p^\prime-1/2)}.
$$
On the other hand, we have
\begin{align*}
\| f \|_{L^{p^\prime}_s} = \| U_\lambda g \|_{L^{p^\prime}_s}
\approx \lambda^{-n/p^\prime} \lambda^s. 
\end{align*}
Combining these two estimates, we obtain
$$
\lambda^{s 
-n/p^\prime -n \mu_2(p^\prime,q^\prime) - \alpha n(1/p^\prime-1/2)}  \gtrsim  1
$$
for all $\lambda \geq 1$.
Letting $\lambda \to \infty$
yields the necessary condition
\begin{align*}
s &\geq \alpha n(1/p^\prime-1/2)
+n/p^\prime +n (-2/p^\prime+1/q^\prime) \\
&= (\alpha -2) n(1/p^\prime -1/2)
+ 2n (1/p^\prime-1/2) +n/p^\prime +n (-2/p^\prime+1/q^\prime)\\
&= (\alpha -2)n(1/p^\prime-1/2)
+n(1/p^\prime+1/q^\prime -1) \\
&= (\alpha -2)n(1/2-1/p) + n(1-1/p -1/q) \\
&= (\alpha -2) n|1/p-1/2| -n \nu_2(p,q),
\end{align*}
since $\mu_2(p^\prime,q^\prime)=-2/p^\prime +1/q^\prime$
and $\nu_2(p,q)=1/p+1/q-1$.
\end{proof}

\noindent
(iv) Let $2 \leq p \leq \infty$ and $p <q$.
Contrary to our claim,
suppose that
there exists $\varepsilon >0$
such that 
$s= (\alpha -2)n|1/p-1/2| -n \nu_2(p,q) -\varepsilon$
implies \eqref{main estimate}.
Then, by interpolation with the estimate
for a point $Q(1/p_1,1/q_1)$ with $2<p_1<\infty, p^\prime_1 < q_1  < p_1$ and $s=(\alpha -2)n|1/p_1-1/2|
-n \nu_2(p_1,q_1)$ (which holds by Corollary \ref{MAIN THEOREM 4}),
one would obtain an improved estimates of the segment
joining $P(1/p,1/q)$ and $Q(1/p_1,1/q_1)$,
which is not possible.
In the same way as above,
we can treat the case $1 \leq p \leq 2$ and $p^\prime < q$,
and we have the conclusion. 

$$
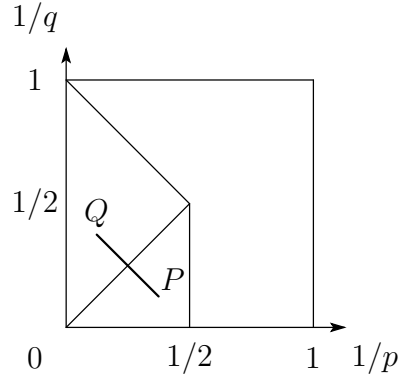
\unitlength 0.1in
\begin{picture}( 18.4300, 17.8000)(  0.0400,-19.0600)
%
\special{pn 8}%
\special{pa 392 1830}%
\special{pa 392 372}%
\special{fp}%
\special{sh 1}%
\special{pa 392 372}%
\special{pa 372 440}%
\special{pa 392 426}%
\special{pa 412 440}%
\special{pa 392 372}%
\special{fp}%
\special{pa 392 1830}%
\special{pa 1848 1830}%
\special{fp}%
\special{sh 1}%
\special{pa 1848 1830}%
\special{pa 1780 1810}%
\special{pa 1794 1830}%
\special{pa 1780 1850}%
\special{pa 1848 1830}%
\special{fp}%
%
\special{pn 8}%
\special{pa 1686 1830}%
\special{pa 1686 534}%
\special{fp}%
\special{pa 1686 534}%
\special{pa 392 534}%
\special{fp}%
\special{pa 392 534}%
\special{pa 1038 1182}%
\special{fp}%
\special{pa 1038 1182}%
\special{pa 1038 1830}%
\special{fp}%
\special{pa 1038 1182}%
\special{pa 392 1830}%
\special{fp}%
\put(2.2900,-19.9100){\makebox(0,0){$0$}}%
\put(10.3700,-19.9100){\makebox(0,0){$1/2$}}%
\put(16.8500,-19.9100){\makebox(0,0){$1$}}%
\put(20.0800,-19.9100){\makebox(0,0){$1/p$}}%
\put(2.2900,-11.8200){\makebox(0,0){$1/2$}}%
\put(2.2900,-5.3400){\makebox(0,0){$1$}}%
\put(2.2900,-2.1100){\makebox(0,0){$1/q$}}%
\put(5.5200,-12.3100){\makebox(0,0){$Q$}}%
\put(9.4900,-15.7800){\makebox(0,0){$P$}}%
%
\special{pn 13}%
\special{pa 876 1668}%
\special{pa 552 1344}%
\special{fp}%
\end{picture}%
$$

\end{proof}

\subsection{Pseudo-differential operator}

Let $\sigma(x,\xi)$ be a function
on ${\mathbf R}^n \times {\mathbf R}^n$.
Then the pseudo-differential operator $\sigma^W(X,D)$
is  defined by
$$
\sigma^W(X,D)f(x) = \frac{1}{(2 \pi)^n} \iint_{{\mathbf R}^{2n}}
\sigma \left( \frac{x+y}{2},\xi \right) e^{i(x-y) \cdot \xi} f(y) dyd \xi.
$$
We give a condition for $\sigma^W(X,D)$ to be bounded
from $M^{p,q}({\mathbf R}^n)$ to $L^p_s({\mathbf R}^n)$.
To this end, we recall the following fundamental results about 
$\sigma^W(X,D)$
with $\sigma \in M^{\infty,1}({\mathbf R}^{2n})$.

\begin{previousthmD}[Sj\"ostrand \cite{Sjostrand 94},\cite{Sjostrand}]
Let $\sigma \in M^{\infty,1}({\mathbf R}^{2n})$.
Then $\sigma^W(X,D)$ is bounded on $L^2({\mathbf R}^n)$.
Moreover, if $\sigma^W(X,D)$ is invertible on $L^2({\mathbf R}^n)$,
then $(\sigma^W(X,D))^{-1}= \tau^W(X,D)$ for
some $\tau \in M^{\infty,1}({\mathbf R}^{2n})$.

\end{previousthmD}

\begin{previousthmE}[Gr\"ochenig-Heil \cite{Grochenig-Heil}]
Let $1 \leq p,q \leq \infty$ and $\sigma \in M^{\infty,1}({\mathbf R}^{2n})$. Then
$\sigma^W(X,D)$ is bounded on $M^{p,q}({\mathbf R}^n)$.
\end{previousthmE}

Now, we state our result which is an analog of
Theorems \ref{THEOREM 0<a<2} and \ref{THEOREM 2<a<infty}.

\begin{theorem}\label{Last Theorem}
Let $1 \leq p,q \leq \infty, s \in {\mathbf R}^n$
and $\sigma \in M^{\infty,1}({\mathbf R}^{2n})$.
\begin{enumerate}
\item[$(1)$]  If $M^{p,q}({\mathbf R}^n) \hookrightarrow L^p_s({\mathbf R}^n)$,
then $\sigma^W(X,D)$ is bounded from $M^{p,q}({\mathbf R}^n)$ to $L^p_s({\mathbf R}^n)$.

\item[$(2)$]  Suppose that $\sigma^W(X,D)$ is invertible on $L^2({\mathbf R}^n)$.
If $\sigma^W(X,D)$ is bounded from
$M^{p,q}({\mathbf R}^n)$ to $L^p_s({\mathbf R}^n)$,
then $\mathcal{M}^{p,q}({\mathbf R}^n)
 \hookrightarrow L^p_s({\mathbf R}^n)$.

\end{enumerate}

\end{theorem}

\begin{proof}
(1) 
Since $\sigma^W(X,D)$ is bounded on $M^{p,q}({\mathbf R}^n)$ by Theorem E,
we have
$$
\| \sigma^W(X,D) f \|_{L^p_s} \lesssim \| \sigma^W(X,D)f \|_{M^{p,q}}
\lesssim \| f \|_{M^{p,q}},
$$
which means $\sigma^W(X,D)$ is bounded from $M^{p,q}({\mathbf R}^n)$ to $L^p_s({\mathbf R}^n)$.

\noindent
(2) Since $(\sigma^W(X,D))^{-1}
= \tau^W(X,D)$ for some $\tau \in M^{\infty,1}({\mathbf R}^{2n})$
and $\tau^W(X,D)$ is bounded on $M^{p,q}({\mathbf R}^n)$ by Theorems D and E,
we obtain
$$
\| f \|_{L^p_s} = \| \sigma^W(X,D) \tau^W(X,D) f \|_{L^p_s}
\lesssim \|  \tau^W(X,D) f \|_{M^{p,q}}
\lesssim \| f \|_{M^{p,q}}
$$
for $f\in{\mathcal{S}}({\mathbf{R}}^n)\subset L^2({\mathbf{R}}^n)$,
which means $\mathcal{M}^{p,q}({\mathbf R}^n)
 \hookrightarrow L^p_s({\mathbf R}^n)$.

\end{proof}

If we combine Theorem \ref{Last Theorem} and Theorem \ref{MAIN THEOREM},
we have similar results to Corollaries \ref{cor5.2} and \ref{MAIN THEOREM 4}.
We will not state them explicitly but it is just a straightforward task.





\bibliographystyle{elsarticle-num}
\bibliography{<your-bib-database>}







\end{document}